\documentclass{amsart}%
\usepackage{cite}
\usepackage[hypertex]{hyperref}
\usepackage{amssymb}
\usepackage{amsmath}
\usepackage{amsthm}
\usepackage{amsfonts}
\usepackage{graphicx}
\usepackage{bbm}
\usepackage{color}
\usepackage[toc,page]{appendix}

\newtheorem{thm}{Theorem}[section]
\newtheorem{cor}[thm]{Corollary}

\newtheorem{assumption}{Assumption}

\newtheorem{nota}[thm]{Notation}
\newtheorem{prop}[thm]{Proposition}

\theoremstyle{definition}
\newtheorem{df}[thm]{Definition}

\newtheorem{rem}[thm]{Remark}
\newtheorem{notation}[thm]{Notation}
\numberwithin{equation}{section}
\begin{document}

\title[An application of a functional inequality]{An application of a functional inequality to quasi-invariance in infinite dimensions}

\author[Gordina]{Maria Gordina{$^{\dag}$}}
\thanks{\footnotemark {$\dag$} Research was supported in part by NSF Grant DMS-1007496.}
\address{$^{\dag}$ Department of Mathematics\\
University of Connecticut\\
Storrs, CT 06269,  U.S.A.}
\email{maria.gordina@uconn.edu}



%\date{\today \ \emph{File:\jobname{.tex}}}

\begin{abstract} One way to interpret smoothness of a measure in infinite dimensions  is quasi-invariance of the measure under a class of transformations. Usually such settings lack a reference measure such as the Lebesgue or Haar measure, and therefore we can not use smoothness of a density with respect to such a measure. We describe how a functional inequality can be used to prove quasi-invariance results in several settings. In particular, this gives a different proof of the classical Cameron-Martin (Girsanov) theorem for an abstract Wiener space. In addition, we revisit several more geometric examples, even though the main abstract result concerns  quasi-invariance of a measure under a group action on a measure space.

\end{abstract}

\maketitle

\tableofcontents

\renewcommand{\contentsname}{Table of Contents}

\section{Introduction}\label{s.1}

Our goal in this paper is to describe how a functional inequality can be used to prove quasi-invariance of certain measures in infinite dimensions. Even though the original argument was used in a geometric setting, we take a slightly different approach in this paper. Namely, we formulate a method that can be used to prove quasi-invariance of a measure under a group action.

Such methods are useful in infinite dimensions when usually there is no a natural reference measure such as the Lebesgue measure. At the same time quasi-invariance of measures is a useful tool in proving regularity results when it is reformulated as an integration by parts formula. We do not discuss significance of such results, and moreover we do not refer to the extensive literature on the subject, as it is beyond the scope of our paper.

We start by describing an abstract setting of how finite-dimensional approximations can be used to prove such a quasi-invariance. In \cite{DriverGordina2009} this method was applied to projective and inductive limits of finite-dimensional Lie groups acting on themselves by left or right multiplication. In that setting a functional inequality (integrated Harnack inequality) on the finite-dimensional approximations leads to a quasi-invariance theorem on the infinite-dimensional group space. Similar methods were used in the elliptic setting  on infinite-dimensional Heisenberg-like groups in \cite{DriverGordina2008}, and on semi-infinite Lie groups in \cite{Melcher2009}. Note that the assumptions we make below in Section \ref{s.fin-dimen} have been verified in these settings, including the sub-elliptic case for infinite-dimensional Heisenberg group in \cite{BaudoinGordinaMelcher2013}. Even though the integrated Harnack inequality we use in these situations have a distinctly geometric flavor, we show in this paper that it does not have to be.

The paper is organized as follows. The general setting is described in Section \ref{s.2} and \ref{s.fin-dimen}, where Theorem \ref{t.qi} is the main result. One of the ingredients for this result is quasi-invariance for finite-dimensional approximations which is described in Section \ref{s.fin-dimen}. We review the connection between an integrated Harnack inequality and Wang's Harnack inequality in Section \ref{s.funineq}. Finally, Section \ref{s.examples} gives several examples of how one can use Theorem \ref{t.qi}. We describe in detail the case of an abstract Wiener space, where the group in question is identified with the Cameron-Martin subspace acting by translation on the Wiener space. In addition we discuss elliptic (Riemannian) and sub-elliptic (sub-Riemannian) infinite-dimensional groups which are examples of a subgroup acting on the group by multiplication.

\subsection*{Acknowledgement} The author is grateful to Sasha Teplyaev and Tom Laetsch for useful discussions and helpful comments.

\section{Notation}\label{s.2}

Suppose $G$ is a topological group with the identity $e$, $X$ is a topological space, $\left( X, \mathcal{B}, \mu \right)$ is a measure space, where $\mathcal{B}$ is the Borel $\sigma$-algebra, and $\mu$ is a probability measure. We assume that $G$ is endowed with the structure of a Hilbert Lie group (e.g. \cite{delaHarpeLNM}), and further that its Lie algebra $\mathfrak{g}:= Lie \left( G \right) = T_{e}G$ is equipped with a Hilbertian inner product, $\langle \cdot, \cdot \rangle$. The corresponding distance on $G$ is denoted by $d\left( \cdot, \cdot \right)$. In addition, we assume that $G$ is separable, and therefore we can use what is known about Borel actions of Polish groups \cite{Becker1998, BeckerKechris1993}. Once we have an inner product on the Lie algebra $\mathfrak{g}$, we can define the length of a path in $G$ as follows. Suppose $k \in C^{1}\left( [0, 1], G \right)$, $k\left( 0 \right)=e$, then

\begin{equation}\label{e.2.1}
l_{G}\left( k \left( \cdot \right) \right):= \int_{0}^{\cdot} \vert L_{ k\left(t \right)^{-1} \ast} \dot{k}\left( t \right)\vert dt,
\end{equation}
where $L_{g}$ is the left translation by $g \in G$.

We assume that $G$ acts measurably on $X$, that is, there is a (Borel) measurable map $\Phi: G \times X \longrightarrow X$ such that

\begin{align*}
& \Phi\left( e, x \right)=x, \text{ for all } x \in X,
\\
& \Phi\left( g_{1}, \Phi\left( g_{2}, x \right) \right)=\Phi\left( g_{1}g_{2}, x \right), \text{ for all } x \in X, g_{1}, g_{2} \in G.
\end{align*}
We often will use $\Phi_{g}:=\Phi\left( g, \cdot \right)$ for $g \in G$.
\begin{df} Suppose $\Phi$ is  a measurable group action of $G$ on $X$.
\begin{enumerate}

  \item In this case we denote by $\left(\Phi_{g}\right)_{\ast}\mu$ the \emph{pushforward measure} defined by

\[
\left(\Phi_{g}\right)_{\ast}\mu\left( A \right):=\mu\left( \Phi\left( g^{-1}, A \right) \right), \text{ for all } A \in \mathcal{B}, g \in G;
\]
  \item the measure $\mu$ is \emph{invariant} under the action $\Phi$ if

  \[
  \left(\Phi_{g}\right)_{\ast}\mu=\mu \text{ for all }  g \in G;
  \]
  \item the measure $\mu$ is \emph{quasi-invariant} with respect to the action $\Phi$ if $\left(\Phi_{g}\right)_{\ast}\mu$ and $\mu$ are mutually absolutely continuous for all $g \in G$.
\end{enumerate}

\end{df}

\begin{nota} For a topological group $G$ acting measurably on the measure space $\left( X, \mathcal{B}, \mu \right)$ in such a way that $\mu$ is quasi-invariant under the action by $G$, the Radon-Nikodym derivative of  $\left(\Phi_{g}\right)_{\ast}\mu$ with respect to $\mu$ is denoted by

\[
J_{g}\left( x \right):=\frac{\left(\Phi_{g}\right)_{\ast}\mu \left( dx \right)}{\mu \left( dx \right)} \text{ for all } g \in G, x \in X.
\]
\end{nota}
For  a thorough discussion of the Radon-Nikodym derivative in this setting we refer to \cite[Appendix D]{BuehlerPhD}

\section{Finite-dimensional approximations and quasi-invariance}\label{s.fin-dimen}

We start by describing approximations to both the group $G$ and the measure space $\left( X, \mathcal{B}, \mu \right)$. At the end we also need to impose certain conditions to have consistency of the group action defined on these approximations. As $X$ is a topological space, we denote by $C_{b}\left( X \right)$ the space of continuous bounded real-valued functions.

\begin{assumption}[Lie group assumptions]\label{a.1} Suppose  $\left\{ G_{n}\right\}_{n \in \mathbb{N}}$ is a collection of finite-dimensional unimodular Lie subgroups of $G$ such that $G_{n} \subset G_{m}$ for all $n < m$. We assume that there exists a smooth section $\left\{ s_{n}: G \longrightarrow G_{n} \right\}_{n \in \mathbb{N}}$, that is, $s_{n} \circ i_{n}=id_{G_{n}}$, where $i_{n}: G_{n} \longrightarrow G$ is the smooth injection. We suppose that $\bigcup_{n \in \mathbb{N}} G_{n}$ is a dense subgroup of $G$. In addition, we assume that the length of a path in $G$ can be approximated by the lengths in $G_{n}$, namely, if $k \in C^{1}\left( [0, 1], G \right)$, $k\left( 0 \right)=e$, then

\begin{equation}\label{e.3.4}
l_{G}\left( k \left( \cdot \right) \right)=
\lim_{n \to \infty}
l_{G_{n}}\left( s_{n} \left( k \left( \cdot \right) \right)\right).
\end{equation}
\end{assumption}
Note that $s_{n}$ does not have to be a group homomorphism.
\begin{assumption}[Measure space assumptions]\label{a.2} We assume that $X$ is a separable topological space with a sequence of topological spaces $X_{n} \subset X$ which come with corresponding continuous maps $\pi_{n}: X \longrightarrow X_{n}$ satisfying the following properties. For any $f \in C_{b}\left( X \right)$

\begin{equation}\label{e.3.1}
\int_{X} f d\mu=\lim_{n \to \infty} \int_{X_{n}} f \circ j_{n} d\mu_{n},
\end{equation}
where $j_{n}: X_{n} \longrightarrow X$ is the continuous injection map, and $\mu_{n}$ is the pushforward measure $\left( \pi_{n}  \right)_{\ast} \mu$.
\end{assumption}
Our last assumption concerns the group action for these approximations.

\begin{assumption}[Group action assumptions]\label{a.3} The approximations to group $G$ and the measure space $\left( X, \mu \right)$ are consistent with the group action in the following way

\begin{align*}
& \Phi\left( G_{n} \times X_{n}\right) \subset X_{n} \text{ for each } n \in \mathbb{N},
\\
&  \Phi_{g}: X \to X \text{ is a continuous map }  \text{ for each } g \in G.
\end{align*}
\end{assumption}
We denote by  $\Phi^{n}$ the restriction of $\Phi$ to $G_{n} \times X_{n}$. Observe that $\Phi^{n}=\Phi\circ \left( i_{n}, j_{n} \right)$ which together with Assumption \ref{a.3}, it is clear that $\Phi^{n}$ is a measurable group action of $G_{n}$ on $\left( X_{n}, \mathcal{B}_{n}, \mu_{n} \right)$.

Suppose now that $\mu_{n}$ is quasi-invariant under the group action $\Phi^{n}$, and let $J_{g}^{n}$ be the Radon-Nikodym derivative $\left( \Phi^{n}_{g}\right)_{\ast}\mu_{n}$ with respect to $\mu_{n}$. We assume that there is a positive constant  $C=C\left( p \right)$ such that for any $p \in [1, \infty)$ and $g \in G_{n}$
\begin{equation}
\label{e.4.4}
\|J_{g}^{n}\|_{L^{p}( X_{n}, \mu_{n})} \leqslant
	\exp \left( C\left( p \right) d_{G_{n}}^{2}\left( e, g \right)\right).
\end{equation}
Note that the constant $C\left( p \right)$ does not depend on $n$.

\begin{rem} The fact that this estimate is Gaussian (with the square of the distance) does not seem to be essential. But as we do not have examples with a different exponent, we leave \eqref{e.4.4} as is. Moreover, we could consider a more general function on the right hand side than an exponential of the distance squared.
\end{rem}

\begin{thm}[Quasi-invariance of $\mu$]\label{t.qi} Suppose we have a group $G$ and a measure space $\left( X, \mathcal{B}, \mu \right)$ satisfying Assumptions \ref{a.1}, \ref{a.2} and \ref{a.3}, and the uniform estimate \eqref{e.4.4} on the Radon-Nikodym derivatives holds.

Then for all $g \in G$  the measure $\mu$ is quasi-invariant under the action $\Phi_{g}$. Moreover, for all $p \in (1,\infty)$,
\begin{equation}
\label{e.RNest}
\left\|\frac{d \left( \Phi_{g} \right)_{\ast} \mu}{d\mu}\right\|_{L^{p}( X,\mu)}
	\leqslant \exp \left( C\left( p\right) d_{G}^{2}\left( e, g \right)\right).
\end{equation}

\end{thm}

\begin{proof}

Using  \eqref{e.4.4} we see that for any bounded continuous function $f \in C_{b}\left( X \right)$, $n \in \mathbb{N}$, and $g \in G$
\begin{align*}
&\int_{X_{n}} \vert (f\circ j_{n})(\Phi^{n}_{s_{n}\left( g \right)}\left( x \right))\vert d\mu_{n}(x)
	= \int_{X_{n}} J_{s_{n}\left( g \right)}^{n}(x)\vert (f\circ j_{n})(x)\vert d\mu_{n}(x) \\
	&\quad\leqslant \Vert f\circ j_{n}\Vert_{L^{p^{\prime}}(X_{n}, \mu_{n})}  \exp \left( C\left( p \right) d_{G_{n}}^{2}\left( e, s_{n}\left( g \right) \right)\right),
\end{align*}
where $p^{\prime}$ is the conjugate exponent to $p$. Note that by Assumption \ref{a.3}  and definitions of $j_{n}$ and $\Phi^{n}$ for all $\left( g, x \right)  \in G_{n} \times X_{n}$

\[
j_{n}\left(\Phi^{n}_{g}\left( x \right)\right)=\Phi^{n}_{g}\left( x \right)=\Phi_{g}\left( x\right)=\Phi_{g}\left( j_{n} \left( x \right)\right)
\]
and therefore

\[
f\circ j_{n}(\Phi^{n}_{g}\left( x \right))=f\left(\Phi_{g}^{n}\left( x \right) \right)= f\circ \Phi_{g}\left(  j_{n}\left( x \right) \right), \left( g, x \right)  \in G_{n} \times X_{n}.
\]
Thus
\[
\int_{X_{n}} \vert \left( f\circ j_{n} \right)\left( \Phi^{n}_{g}\left( x \right) \right)\vert d\mu_{n}(x)=
\int_{X_{n}} \vert \left( f\circ \Phi_{g} \right)\left(  j_{n}\left( x \right) \right)\vert d\mu_{n}(x).
\]
Allowing $n \rightarrow \infty$ in the last identity and using \eqref{e.3.1} and the fact that $f\circ \Phi_{g} \in C_{b}\left( X \right)$ yields
\begin{equation}\label{e.5.3} \int_{X} \vert f \left(\Phi_{g}\left(  x  \right)\right) \vert d\mu(x)
	\leqslant \Vert f\Vert_{L^{p^{\prime}}(X, \mu)} \exp \left( C\left( p \right) d_{G_{n}}^{2}\left( e, g \right)\right), \text{ for all } g \in G_{n}.
\end{equation}
Thus, we have proved that \eqref{e.5.3} holds for $f \in C_{b}\left( X \right)$ and $g \in G_{n}$. Now we would like to prove \eqref{e.5.3} for the distance $d_{G}$ instead of $d_{G_{n}}$ with $g$ still in $G_{n}$. Take any path $k \in C^{1}\left( [0, 1], G \right)$ such that $k\left( 0 \right)=e$ and $k\left( 1 \right)=g$, and observe that then $s_{n} \circ k \in C^{1}\left( [0, 1], G_{n} \right)$  and therefore \eqref{e.5.3} holds with $d_{G_{n}}^{2}\left( e, g \right)$ replaced by $l_{G_{n}}\left( s_{n} \circ k\right)\left( 1 \right)$. Now we can use \eqref{e.3.1} in Assumption \ref{a.1} and optimizing over all such paths $k$ to see that

\begin{equation}\label{e.3.6} \int_{X} \vert f \left(\Phi_{g}\left(  x  \right)\right) \vert d\mu(x)
	\leqslant \Vert f\Vert_{L^{p^{\prime}}(X, \mu)} \exp \left( C\left( p \right) d_{G}^{2}\left( e, g \right)\right), \text{ for all } g \in \bigcup_{n \in \mathbb{N}}G_{n}.
\end{equation}
By Assumption \ref{a.1}  this union is dense in $G$, therefore dominated convergence along with the continuity of $d_{G}\left( e, g \right)$ in $g$ implies that \eqref{e.3.6} holds for all $g\in G$. Since the bounded continuous functions are dense in $L^{p^{\prime}}( X, \mu)$ (see for example \cite[Theorem A.1, p. 309]{Janson1997}), the inequality in \eqref{e.3.6} implies that the linear functional $\varphi_{g}: C_{b}(X) \rightarrow\mathbb{R}$ defined by
\[
\varphi_{g}(f) := \int_{X} \vert f \left(\Phi_{g}\left(  x  \right)\right) \vert d\mu(x)
\]
has a unique extension to an element, still denoted by $\varphi_{g}$, of $L^{p^{\prime}}( X,\mu)^{\ast}$ which satisfies the bound
\[ \vert \varphi_{g}(f)\vert \leqslant \|f\|_{L^{p^{\prime}}( X,\mu)}
	\exp \left( C\left( p\right) d_{G}^{2}\left( e, g \right)\right)
\]
for all $f\in L^{p^{\prime}}( X,\mu)$.  Since $L^{p^{\prime}}( X,\mu)^{\ast}\cong L^{p}( W,\mu)$, there exists a function $J_{g}\in L^{p}( X,\mu)$ such that
\begin{equation}
\label{e.d}
\varphi_{g}(f) = \int_{X} f(x)J_{g}(x)d\mu(x),
\end{equation}
for all $f\in L^{p^{\prime}}( X,\mu)$, and

\[
\|J_{g}\|_{L^{p}( X,\mu)} \leqslant \exp \left( C\left( p \right) d_{G}^{2}\left( e, g \right)\right).
\]

Now restricting (\ref{e.d}) to $f\in C_{b}\left( X \right)$, we may rewrite this equation as
\begin{equation}
\label{e.last}
\int_{X}  f \left(\Phi_{g}\left(  x  \right)\right) d\mu(x)
	= \int_{W} f(x) J_{g}(x) d\mu(x).
\end{equation}
Then a monotone class argument (again use \cite[Theorem A.1]{Janson1997}) shows that \eqref{e.last} is valid for all bounded measurable functions $f$ on $W$.  Thus, $d \left( \Phi_{g} \right)_{\ast} \mu / d\mu$ exists and is given by $J_{g}$, which is in
$L^p$ for all $p\in(1,\infty)$ and satisfies the bound \eqref{e.RNest}.

\end{proof}

\section{A functional inequality}\label{s.funineq}

In this section we would like to revisit an observation made in \cite{DriverGordina2009}. Namely, \cite[Lemma D.1]{DriverGordina2009} connects Wang's Harnack inequality with an estimate similar to \ref{e.4.4}. It is easy to transfer this argument from the setting of Riemannian manifolds to a more general situation.

We start with an integral operator on $L^{2}\left( X, \nu \right)$, where $\left( X, \nu \right)$ is a $\sigma$-finite measure space. Namely, let

\[
Tf\left( x \right):=\int_{X} p\left( x, y \right) f\left( y \right) d\nu\left( y \right), f \in L^{2}\left( X, \nu \right),
\]
where the integral kernel $p\left(  x, y\right)$ is assumed to satisfy the following properties.

\begin{align*}
& \text{ positive } & p\left(  x, y\right)>0 \text{ for all } x, y \in X,
\\
& \text{ conservative } & \int_{X} p\left( x, y \right) d\nu\left( y \right)=1  \text{ for all } x \in X,
\\
& \text{ symmetric } & p\left(  x, y\right)=p\left(  y, x\right)  \text{ for all } x, y \in X,
\\
&  \text{ continuous } & p\left(  \cdot , \cdot \right):  X \times X \longrightarrow \mathbb{R} \text{ is continuous}.
\end{align*}

Some of these assumptions might not be needed for the proof of Proposition \ref{p.4.1}, but we make them to simplify the exposition. Note that in our applications this integral kernel is the heat kernel for  a strongly  continuous, symmetric, Markovian semigroup in $L^{2}\left( X, \nu  \right)$, therefore the corresponding heat kernel is positive, symmetric with the total mass not exceeding $1$, in addition to having the semigroup property or being the approximate identity in $L^{2}\left( X, \nu  \right)$.  In our examples this heat semigroup is also conservative, therefore the heat kernel is conservative (stochastically complete), and thus $p\left( x, y \right) d\nu\left( y \right)$ is a probability measure.

The following proposition is a generalization of  \cite[Lemma D.1]{DriverGordina2009}, and it simply reflects the fact that $\left(  L^{p}\right)^{\ast}$ and $L^{p^{\prime}}$ are isometrically isomorphic Banach spaces for $1<p<\infty$ and $p^{\prime}=p/\left(  p-1\right)$, the conjugate exponent to $p$.

\begin{prop}\label{p.4.1}
Let $x,y \in X$, $p \in (1,\infty)$ and $C \in(0,\infty]$ which might depend on $x$ and $y$. Then

\begin{equation}\label{e.D.1}
\left[  \left(  Tf\right)  \left(  x\right)  \right]^{p} \leqslant
C^{p}\left(  Tf^{p}\right)  \left(  y\right)  \text{ for all } f\geqslant 0
\end{equation}
if and only if
\begin{equation}\label{e.D.2}
\left(  \int_{X}\left[  \frac{p\left(  x, z\right)  }{p \left(
y, z\right)  }\right]^{p^{\prime}}p\left(  y,z\right)  d\nu\left(
z\right)  \right)^{1/p^{\prime}}\leqslant C.
\end{equation}

\end{prop}

\begin{proof}
Since $p \left(  \cdot, \cdot \right)$ is positive, we can write
\[
\left(  Tf\right)  \left(  x \right)  = \int_{X}\frac{p\left(
x, z\right)  }{p\left(  y, z\right)  } f\left(  z\right)  p\left(
y, z\right)  d\nu\left(  z\right).
\]
We denote $d\mu_{y}\left(  \cdot \right):=p\left(  y, \cdot\right)  d\nu\left(  \cdot\right)$ and
$g_{x, y }\left(  \cdot \right):=\frac{p\left(  x, \cdot\right)  }{p\left( y, \cdot \right)}$, then
\begin{equation}\label{e.D.3}
\left(  Tf\right)  \left(  x\right) =\int_{X}f\left(  z\right)  g_{x, y }\left(
z\right)  d\mu_{y }\left(  z \right).
\end{equation}
Since $g_{x, y } \geqslant 0$ and $L^{p}\left(  \mu \right)^{\ast}$ is isomorphic to $L^{p^{\prime}}\left(  \mu \right)$, the pairing in \eqref{e.D.3} implies that
\[
\left\Vert g_{x, y } \right\Vert_{L^{p^{\prime}}\left(  \mu\right)  }=\sup_{f\geqslant
0}\frac{\int_{X} f\left(  z\right)  g_{x, y }\left(
z\right)  d\mu_{y }\left(  z \right)
}{\left\Vert f\right\Vert _{L^{p}\left(  \mu_{y}\right)  }} = \sup_{f \geqslant 0}
\frac{\left(  Tf\right)  \left(  x\right)  }{\left[  \left(  T f^{p}\right)  \left(  y\right)  \right]^{1/p}}.
\]
The last equation may be written more explicitly as
\[
\left(  \int_{X}\left[  \frac{p\left(  x, z\right)  }{p\left(
y, z\right)  }\right]^{p^{\prime}}p\left(  y, z\right)  d\nu\left(
z \right)  \right)  ^{1/p^{\prime}}=\sup_{f\geqslant 0}\frac{\left(  Tf\right)
\left(  x\right)  }{\left[  \left(  Tf^{p}\right)  \left(  y\right)
\right]^{1/p}},
\]
and from this equation the result follows.
\end{proof}
\begin{rem} In the case when $T$ is a Markov semigroup $P_{t}$ and $p\left( \cdot, \cdot \right)=p_{t}\left( \cdot, \cdot \right)$ is the corresponding integral kernel, Proposition \ref{p.4.1} shows that a Wang's Harnack inequality is equivalent to an integrated Harnack inequality. Subsection \ref{ss.Harnack} gives more details on this equivalence for Riemannian manifolds, see Corollary \ref{c.5.11}.
\end{rem}

\begin{rem} The connection between Proposition \ref{p.4.1} and \eqref{e.4.4} can be seen if we choose $x$ and $y$ in \eqref{e.D.2} as the endpoints of the group action as follows. Let $x, y \in X$ and $g \in G$ be such that $\Phi_{e}\left( x \right)=x$ and $\Phi_{g}\left( x \right)=y$, then to apply Proposition \ref{p.4.1} we can take the constant in \eqref{e.D.2} to be equal to

\[
\exp \left( C\left( p-1 \right) d_{G_{n}}^{2}\left( e, g \right)\right).
\]
Here the measure on $X$ is $d\mu_{x}\left( z\right)= p\left(  x, z\right)  d\nu\left( z\right)$.
\end{rem}

\section{Examples}\label{s.examples}

\subsection{Abstract Wiener space}
Standard references on basic facts on the Gaussian measures include \cite{BogachevGaussianMeasures, KuoBook1975}. Let $\left(  H, W, \mu \right)$ be an abstract Wiener space, that is, $H$ is a real separable Hilbert space densely continuously embedded into a real separable Banach space $W$, and $\mu$ is the Gaussian measure defined by the characteristic functional

\[
\int_{W}e^{i\varphi\left(  x\right)}d\mu\left(  x\right)  =\exp\left(
-\frac{|\varphi|_{H^{\ast}}^{2}}{2}\right)
\]
for any $\varphi\in W^{\ast}\subset H^{\ast}$. We will identify $W^{\ast}$ with a dense subspace of $H$ such that for any $h \in W^{\ast}$ the linear functional $\langle\cdot, h\rangle$ extends continuously from $H$ to $W$. We will usually write $\langle\varphi, w\rangle:=\varphi\left(  w\right)$ for $\varphi\in W^{\ast}$, $w\in W$. More details can be found in \cite{BogachevGaussianMeasures}. It is known that $\mu$ is a Borel measure, that is, it is defined on the Borel $\sigma$-algebra $\mathcal{B}\left(  W\right)$ generated by the open subsets of $W$.

We would like to apply the material from Sections \ref{s.funineq} \ref{s.fin-dimen} with $\left( X, \mu \right)=\left( W, \mu \right)$ and the group $G=E_{W}$ being the group of (measurable) rotations and translations by the elements from the Cameron-Martin subspace $H$. We can view this group as an infinite-dimensional analogue of the Euclidean group.

\begin{nota}
We call an orthogonal transformation of $H$ which is a topological homeomorphism of $W^{\ast}$ a \textbf{rotation} of $W^{\ast}$. The space of all such rotations is denoted by $O\left(  W \right)$. For any $R \in O\left(  W^{\ast}\right)$ its adjoint, $R^{\ast}$, is defined by

\[
\langle\varphi, R^{\ast}w \rangle:=\langle R^{-1}\varphi, w \rangle, \ w \in W, \varphi \in W^{\ast}.
\]

\end{nota}

\begin{prop}\label{p.5.2}
For any $R \in O\left(  W \right)$ the map $R^{\ast}$ is a $\mathcal{B}\left(  W \right)$-measurable map from $W$ to $W$ and

\[
\mu\circ\left(  R^{\ast}\right)^{-1}=\mu.
\]

\end{prop}

\begin{proof}
The measurability of $R^{\ast}$ follows from the fact that $R$ is continuous on $H$. For any $\varphi\in W^{\ast}$

\begin{align*}
&  \int_{W}e^{i\varphi\left(  x\right)  }d\mu \left(  \left(  R^{ \ast }\right)^{-1}x\right)  = \int_{W}e^{i\langle\varphi, x \rangle} d\mu \left( \left(  R^{\ast}\right)^{-1}x\right)  = \int_{W}e^{i\langle\varphi, R^{ \ast }x\rangle} d\mu \left(  x\right)  =
\\
&  \exp\left(  -\frac{|R^{-1}\varphi|_{H^{\ast}}^{2}}{2}\right)  =\exp\left( -\frac{|\varphi|_{H^{\ast}}^{2}}{2}\right) =\int_{W}e^{i\varphi\left( x\right)  } d\mu\left(  x\right)
\end{align*}
since $R$ is an isometry.
\end{proof}

\begin{cor}
\label{c.6.3} Any $R\in O\left(  W \right)$ extends to a unitary map on $L^{2}\left(  W, \mu \right)$.
\end{cor}

\begin{df} The \emph{Euclidean group} $E_{W}$ is a group generated by measurable rotations $R \in O\left(  W \right)$ and translation $T_{h}:W \rightarrow W$, $T_{h}\left( w\right):= w+h$.
\end{df}

To describe finite-dimensional approximations as in Section \ref{s.fin-dimen}  we need to give more details on the identification of $W^{\ast}$ with a dense subspace of $H$. Let $i:H\rightarrow W$ be the inclusion map, and $i^{\ast}:W^{\ast}\rightarrow H^{\ast}$ be its transpose, i.e. $i^{\ast}\ell:=\ell\circ i$ for all $\ell\in W^{\ast}$. Also let
\[
H_{\ast}:=\left\{  h\in H:\left\langle \cdot, h\right\rangle _{H} \in \operatorname{Ran}(i^{\ast})\subset H^{\ast}\right\}
\]
or in other words, $h\in H$ is in $H_{\ast}$ iff $\left\langle \cdot, h\right\rangle _{H}\in H^{\ast}$ extends to a continuous linear functional on $W$. We will continue to denote the continuous extension of $\left\langle \cdot, h\right\rangle _{H}$ to $W$ by $\left\langle \cdot, h\right\rangle_{H}$. Because $H$ is a dense subspace of $W$, $i^{\ast}$ is injective and because $i$ is injective, $i^{\ast}$ has a dense range. Since $h\mapsto \left\langle \cdot, h\right\rangle _{H}$ as a map from $H$ to $H^{\ast}$ is a conjugate linear isometric isomorphism, it follows from the above comments that for any $h \in H$ we have $h\mapsto \left\langle \cdot, h\right\rangle_{H} \in W^{\ast}$ is a conjugate linear isomorphism too, and that $H_{\ast}$ is a dense subspace of $H$.

Now suppose that $P: H\rightarrow H$ is a finite rank orthogonal projection such that $PH\subset H_{\ast}$. Let $\left\{  e_{j}\right\}_{j=1}^{n}$ be an orthonormal basis for $PH$ and $\ell_{j}=\left\langle \cdot,e_{j}\right\rangle_{H}\in W^{\ast}$.  Then we may extend $P$ to a (unique) continuous operator from $W$ $\rightarrow H$ (still denoted by $P$) by letting

\begin{equation}\label{e.3.42}
P_{n}w:=\sum_{j=1}^{n}\left\langle w, e_{j}\right\rangle _{H}e_{j}=\sum_{j=1}^{n}\ell_{j}\left(  w\right)  e_{j} \text{ for all } w\in W.
\end{equation}
As we pointed put in \cite[Equation 3.43]{DriverGordina2008} there exists $C<\infty$ such that
\begin{equation}\label{e.3.43}
\left\Vert Pw\right\Vert _{H}\leqslant C\left\Vert w\right\Vert _{W}\text{ for
all }w\in W.
\end{equation}

\begin{nota}\label{n.3.24} Let $\operatorname*{Proj}\left(  W\right)  $ denote the
collection of finite rank projections on $W$ such that $PW\subset H_{\ast}$
and $P|_{H}:H\rightarrow H$ is an orthogonal projection, i.e. $P$ has the form
given in Equation \eqref{e.3.42}.
\end{nota}

Also let $\left\{  e_{j}\right\}  _{j=1}^{\infty}\subset H_{\ast}$ be an
orthonormal basis for $H.$ For $n \in \mathbb{N}$, define $P_{n}\in
\operatorname*{Proj}\left(  W\right)$ as in Notation \ref{n.3.24}, i.e.
\begin{equation}\label{e.4.3}
P_{n}\left(  w\right)  =\sum_{j=1}^{n}\left\langle w,e_{j}\right\rangle_{H}e_{j}=\sum_{j=1}^{n}\ell_{j}\left(  w\right)  e_{j}\text{ for all }w\in W.
\end{equation}
Then we see that $P_{n}\left|_{H}\right. \uparrow Id_{H}$.

\begin{prop}\label{p.5.5} The Gaussian measure $\mu$ is quasi-invariant under the translations from $H$ and invariant under orthogonal transformations of $H$.
\end{prop}

\begin{proof} The second part of the statement is the content of Proposition \ref{p.5.2}. We now prove quasi-invariance of $\mu$ under translation by elements in $H$. Let $\{ P_{n} \}_{n \in \mathbb{N}}$ be a collection of operators defined by \eqref{e.4.3}for an orthonormal basis $\{ e_{j}\}_{j=1}^{\infty}$ of $H$ such that $\{ e_{j} \}_{j=1}^{\infty}\subseteq H_{\ast}$. Then $H_{n}:=P_{n}\left( H \right)\cong \mathbb{R}^{n}$, and the pushforward measure $\left( P_{n}\right)_{\ast}\mu$ is simply the standard Gaussian measure $p_{n}\left( x \right)dx$ on $H_{n}$. So if we identify the group of translation $G$ with $H$ and $s_{n}:=P_{n}\left|_{H}\right.$, then the group action  is given by $\Phi_{h}\left( w \right):=w+h, w \in W, h \in H$. Note that Assumptions \ref{a.1}, \ref{a.2} and \ref{a.3} are satisfied, where $j_{n}=P_{n}: W \longrightarrow H_{n}$ etc. In particular, if we denote $h_{n}:=P_{n}\left( h \right) \in \mathbb{R}^{n}, h \in H$, then for any measurable function $f:W \longrightarrow \mathbb{R}$ we see that

\[
f \circ P_{n}\left( w+h \right)=f \circ P_{n} \circ \Phi_{h_{n}}\left( w \right)=f \circ \Phi_{h_{n}} \circ P_{n}\left( w \right).
\]
Therefore
\begin{align*}
& \int_{W} f \circ P_{n}\left( w+h \right)
d\mu\left(w\right)=\int_{H_{n}} f\left( x+P_{n}h \right)
p_{n}\left(x\right)dx=
\\
&
\\
&\int_{H_{n}} f\left( x \right)
p_{n}\left(x-h_{n}\right)dx=\int_{H_{n}} f\left( x \right)
\frac{p_{n}\left(x-h_{n}\right)}{p_{n}\left(x\right)}p_{n}\left(x\right)dx
\\
&
\\
&=\int_{H_{n}} f\left( x \right)
J_{h_{n}}\left(x\right) p_{n}\left(x\right)dx.
\end{align*}
Using an explicit form of the Radon-Nikodym derivative $J_{h_{n}}\left(x\right)$, we see that for any $f \in L^{p^{\prime}}\left( W, \mu\right)$

\begin{align*}
\int_{W} \vert f \circ P_{n}\left( w+h \right)\vert
d\mu\left(w\right) & \leqslant\Vert f
\Vert_{L^{p^{\prime}}\left(p_{n}\left(x\right)dx\right)}
\left\Vert
\frac{p_{n}\left(x-h_{n}\right)}{p_{n}\left(x\right)}\right\Vert_{L^{p}\left(p_{n}\left(x\right)dx\right)}
\\
&
\\
& \leqslant \Vert f \circ P_{n}
\Vert_{L^{p^{\prime}}\left(p_{n}\left(x\right)dx\right)}\exp\left(
\frac{\left(p-1\right)\Vert h_{n}\Vert_{H}^{2}}{2}\right).
\end{align*}
Thus \eqref{e.4.4} is satisfied, and therefore Theorem \ref{t.qi} is applicable, which proves the quasi-invariance with the Radon-Nikodym derivative satisfying

\begin{equation}\label{e.5.4}
\Vert J_{h} \Vert_{ L^{p}\left( W, \mu \right)} \leqslant \exp{\left( \frac{\left( p-1 \right)\Vert h \Vert_{H}^{2}}{2} \right)}.
\end{equation}
\end{proof}

\begin{rem} The statement of Proposition \ref{p.5.5} of course follows from the Cameron-Martin theorem which  states that $\mu$ is quasi-invariant under translations by elements in $H$ with the Radon-Nikodym derivative given by

\[
\frac{d\left(  T_{h}\right)_{\ast}\mu}{d\mu}\left(  w\right)
=\frac{d\left(  \mu\circ T_{h}^{-1}\right)  }{d\mu}\left(  w\right)
=\frac{d\left(  \mu\circ T_{-h}\right)  }{d\mu}\left(  w\right)
=e^{-\langle h,w\rangle-\frac{|h|^{2}}{{2}}},\ w\in W, h\in H.
\]
Thus \eqref{e.5.4} is sharp.
\end{rem}

\begin{rem} Following \cite{DriverHall1999a} we see that quasi-invariance of the Gaussian measure $\mu$ induces the Gaussian regular representation of the Euclidean group $E_{W}$ on $L^{2}\left( W, \mu \right)$ by

\begin{align*}
&  \left(  U_{R,h}f\right)  \left(  w\right): =\left(  \frac{d\left(
\mu\circ\left(  T_{h}R^{\ast}\right)  \right)  }{d\mu}\left(  w\right)
\right)^{1/2}f\left(  \left(  T_{h}R^{\ast}\right)^{-1}\left(  w\right)
\right)  =
\\
&  \left(  \frac{d\left(  \mu\circ T_{h}\right)  }{d\mu}\left(
w\right)  \right)^{1/2}f\left(  \left(  R^{\ast}\right)^{-1}\left(
w-h\right)  \right)  =
\\
&  e^{\langle h,w\rangle-\frac{|h|^{2}}{{2}}}f\left(  \left(  R^{\ast}\right)^{-1}\left(  w-h\right)  \right), \ w\in W\nonumber
\end{align*}
which is well-defined by Corollary \ref{c.6.3}. It is clear that this is a
unitary representation.
\end{rem}

\subsection{Wang's Harnack inequality}\label{ss.Harnack} This follows \cite[Appendix D]{DriverGordina2009}. The following theorem appears in \cite{WangFY1997a, WangFY2004a} with $k=-K$,  $V\equiv0$. We will use the following notation

\begin{equation}\label{e.5.8}
c\left( t \right):=
\left\{
\begin{array}{cc}
  \frac{t}{e^{t}-1} & t\not=0,
  \\
  1 & t=0.
\end{array}
\right.
\end{equation}

\begin{thm}[Wang's Harnack inequality]\label{t.D.2} Suppose that $M$ is a complete connected Riemannian manifold such that $\operatorname{Ric}\geqslant kI$ for some $k\in\mathbb{R}$. Then for all $p>1$,  $f\geqslant 0$,  $t>0$,  and $x, y \in M$ we
have
\begin{equation}\label{e.D.4}
\left(  P_{t}f\right)^{p}\left(  y\right)  \leqslant \left(  P_{t}f^{p}\right)
\left(  z \right)  \exp\left(\frac{ p^{\prime}k}{e^{kt}-1}d^{2}\left(
y, z\right)  \right).
\end{equation}

\end{thm}

\begin{cor}\label{c.5.11}
Let $\left(  M,g\right)$ be a complete Riemannian manifold such that $\operatorname{Ric}\geqslant kI$ for some $k \in \mathbb{R}$. Then for every $y,z \in M$ and $p \in\lbrack1,\infty)$
\begin{equation} \label{e.D.5}
\left(  \int_{M}\left[  \frac{p_{t}\left(  y, x\right)  }{p_{t}\left( z, x\right)  }\right]^{p}p_{t}\left(  z, x\right)  dV\left(  x\right) \right)^{1/p}\leqslant \exp\left(  \frac{c\left(  kt\right)  \left( p-1\right) }{2t}d^{2}\left(  y, z\right)  \right)
\end{equation}
where $c\left(  \cdot\right)$ is defined by \eqref{e.5.8}, $p_{t}\left(  x, y\right)$ is the heat kernel on $M$ and $d\left( y, z\right)$ is the Riemannian distance from $x$ to $y$ for $x, y\in M$.
\end{cor}

\begin{proof}
From Lemma \ref{p.4.1} and Theorem \ref{t.D.2} with
\[
C=\exp\left(  \frac{p^{\prime}}{p}\frac{k}{e^{kt}-1}d^{2}\left(  y, z\right)
\right)  =\exp\left(  \frac{1}{p-1}\frac{k}{e^{kt}-1}d^{2}\left(  y, z\right)
\right),
\]
it follows that it follows that
\[
\left(  \int_{M}\left[  \frac{p_{t}\left(  x, z\right)  }{p_{t}\left(
y, z\right)  }\right]^{p^{\prime}}p_{t}\left(  y, z\right)  dV\left(
z\right)  \right)^{1/p^{\prime}}\leqslant \exp\left(  \frac{1}{p-1}\frac{k}
{e^{kt}-1}d^{2}\left(  y, z\right)  \right).
\]
Using $p-1=\left(  p^{\prime}-1\right)^{-1}$ and then interchanging the
roles of $p$ and $p^{\prime}$ gives \eqref{e.D.5}.
\end{proof}

The reason we call \ref{e.D.2} an integrated Harnack inequality on a $d$-dimensional manifold $M$ is as follows. Recall the classical Li--Yau Harnack inequality (\cite{LiYau1986} and  \cite[Theorem 5.3.5]{DaviesHeat_Kernels_and_Spectral_Theory}) which  states that if $\alpha>1$,  $s>0$,  and $\operatorname{Ric}\geqslant -K$ for some $K\geqslant 0$, then
\begin{equation}\label{e.D.6}
\frac{p_{t}\left(  y, x\right)  }{p_{t+s}\left(  z, x\right)  }\leqslant \left(
\frac{t+s}{t}\right)^{d\alpha/2}\exp\left(  \frac{\alpha d^{2}\left(
y, z\right)  }{2s}+\frac{d \alpha Ks}{8\left(  \alpha-1\right)  }\right),
\end{equation}
for all $x, y, z \in M$ and $t>0$.  However, when $s=0$,  \eqref{e.D.6} gives no information on $p_{t}\left(  y, x\right)  /p_{t}\left(  z, x\right)$ when $y\neq z$. This inequality is based on the Laplacian $\Delta/2$ rather than $\Delta$, $t$ and $s$ should be replaced by $t/2$ and $s/2$ when applying the results in \cite{LiYau1986, DaviesHeat_Kernels_and_Spectral_Theory}.

\subsection{Infinite-dimensional Heisenberg-like groups:} \textbf{Riemannian and sub-Riemannian cases.} These examples represent infinite-dimensional versions of the group action of a Lie group on itself by left or right multiplication. The difference is in geometry of the space on which the group acts on: Riemannian and sub-Riemannian. In both cases we proved \eqref{e.4.4}, where the constant $C$ depends on the geometry, and the distance used is Riemannian or Carnot-Carath\'{e}odory.

Let $(W, H,\mu)$ be an abstract Wiener space and let $\mathbf{C}$ be a finite-dimensional inner product space. Define $\mathfrak{g}:=W\times \mathbf{C}$ to be an infinite-dimensional Heisenberg-like Lie algebra, which is constructed as an infinite-dimensional step 2 nilpotent Lie algebra with continuous Lie bracket. Namely, let $\omega: W \times W \rightarrow\mathbf{C}$  be a continuous skew-symmetric bilinear form on $W$.  We will also assume that $\omega$
is surjective.

Let $\mathfrak{g}$ denote $W\times\mathbf{C}$ when thought of as a Lie algebra
with the Lie bracket given by
\begin{equation}
\label{e.3.5}
[(X_1,V_1), (X_2,V_2)] := (0, \omega(X_1,X_2)).
\end{equation}
Let $G$ denote $W\times\mathbf{C}$ when thought of as a group with
multiplication given by
\begin{equation*}
 g_1 g_2 := g_1 + g_2 + \frac{1}{2}[g_1,g_2],
\end{equation*}
where $g_1$ and $g_2$ are viewed as elements of $\mathfrak{g}$. For $g_i=(w_i, c_i)$, this may be written equivalently as
\begin{equation}
\label{e.3.2}
(w_1,c_1)\cdot(w_2,c_2) = \left( w_1 + w_2, c_1 + c_2 +
    \frac{1}{2}\omega(w_1,w_2)\right).
\end{equation}
Then $G$ is a Lie group with Lie algebra $\mathfrak{g}$, and $G$ contains the subgroup $G_{CM} = H\times\mathbf{C}$ which has Lie
algebra $\mathfrak{g}_{CM}$. In terms of Section \ref{s.2} the Cameron-Martin (Hilbertian) subgroup $G_{CM}$ is the group that is acting on the Heisenberg group $G$ by left or right multiplication.

Using Notation \ref{n.3.24} we can define finite-dimensional approximations to $G$ by using $P\in\mathrm{Proj}(W)$. We assume in addition that $PW$ is sufficiently large to satisfy H\"ormander's condition (that is, $\{\omega(A,B):A,B\in PW\}=\mathbf{C}$). For each $P\in\mathrm{Proj}(W)$, we define $G_P:= PW \times\mathbf{C}\subset H_*\times\mathbf{C}$ and a corresponding projection $\pi_P:G\rightarrow G_P$

\[ \pi_P(w,x):= (Pw,x). \]
We will also let $\mathfrak{g}_P=\mathrm{Lie}(G_P) = PW\times\mathbf{C}$. For each $P\in\mathrm{Proj}(W)$, $G_P$ is a finite-dimensional connected unimodular Lie group.

\begin{notation}
(Riemannian and horizontal distances on $G_{CM}$)
\label{n.length}

\begin{enumerate}
\item For $x=(A,a)\in G_{CM}$, let
\[ |x|_{\mathfrak{g}_{CM}}^2 : = \|A\|_H^2 + \|a\|_\mathbf{C}^2. \]
The {\em length} of a $C^1$-path $\sigma:[ 0, 1]\rightarrow
G_{CM}$ is defined as
\[ \ell(\sigma)
	:= \int_0^1 |L_{\sigma^{-1}(s)*}\dot{\sigma}(s)|_{\mathfrak{g}_{CM}} \,ds.
\]
By $C^{1}_{CM}$ we denote the set of paths $\sigma:[0,1]\rightarrow G_{CM}$.
\item \label{i.2}
A $C^1$-path $\sigma:[0, 1]\rightarrow G_{CM}$ is {\em horizontal} if
$L_{\sigma(t)^{-1}*}\dot{\sigma}(t)\in H\times\{0\}$
for a.e.~$t$.  Let $C^{1,h}_{CM}$ denote the set of horizontal paths
$\sigma:[0,1]\rightarrow G_{CM}$.

\item The {\em Riemannain distance} between $x, y\in G_{CM}$ is defined by
\[ d(x,y) := \inf\{\ell(\sigma): \sigma\in C^{1}_{CM} \text{ such
    that } \sigma(0)=x \text{ and } \sigma(1)=y \}. \]

\item The {\em horizontal distance} between $x,y\in G_{CM}$ is defined by
\[ d^{h}(x,y) := \inf\{\ell(\sigma): \sigma\in C^{1,h}_{CM} \text{ such
    that } \sigma(0)=x \text{ and } \sigma(1)=y \}. \]
\end{enumerate}
The Riemannian and horizontal distances are defined analogously on $G_P$ and will be denoted by
$d_P$ and $d^{h}_{P}$ correspondingly.  In particular, for a sequence $\{P_n\}_{n=1}^\infty\subset\mathrm{Proj}(W)$, we will let $G_n:=G_{P_n}$, $d_n:=d_{P_n}$, and  $d^{h}_n:=d^{h}_{P_n}$.
\end{notation}

Now we are ready to define the corresponding heat kernel measures on $G$. We start by considering two Brownian motions on  $\mathfrak{g}$

\begin{align*}
& b_{t}:= \left(  B\left(  t\right), B_{0}\left(  t\right)  \right), t\geqslant 0,
\\
& b^{h}_{t}:= \left(  B\left(  t\right), 0 \left(  t\right)  \right), t\geqslant 0,
\end{align*}
with variance determined by
\begin{multline*}
\mathbb{E}\left[  \left\langle \left(  B\left(  s\right), B_{0}\left(
s\right)  \right),\left(  A, a\right)  \right\rangle_{\mathfrak{g}_{CM}
}\left\langle \left(  B\left(  t\right), B_{0}\left(  t\right)  \right), \left(  C,c\right)  \right\rangle _{\mathfrak{g}_{CM}}\right]
\\
=\operatorname{Re}\left\langle \left(  A, a\right),\left(  C,  c\right)
\right\rangle _{\mathfrak{g}_{CM}}\min\left(  s,t\right)
\end{multline*}
for all $s, t\in \lbrack0, \infty)$, $A, C\in H_{\ast}$ and $a, c\in\mathbf{C}$.

A (Riemannian) \emph{Brownian motion} on $G$ is the continuous $G$--valued process defined by
\begin{equation}
g\left(  t\right)  =\left(  B\left(  t\right), B_{0}\left(  t\right)
+\frac{1}{2}\int_{0}^{t}\omega\left(  B\left(  \tau\right), dB\left(
\tau\right)  \right)  \right).
\end{equation}
Further, for $t>0,$ let $\mu_{t}=\operatorname{Law}\left(  g\left(  t\right) \right)$ be a probability measure on $G$. We refer to $\mu_{t}$ as the time $t$ \emph{ heat kernel measure on } $G$.

Similarly a \emph{horizontal Brownian motion} on $G$ is the continuous $G$--valued process defined by
\begin{equation}
g^{h}\left(  t\right)  =\left(  B\left(  t\right), \frac{1}{2}\int_{0}^{t}\omega\left(  B\left(  \tau\right), dB\left(
\tau\right)  \right)  \right).
\end{equation}
Then for $t>0,$ let $\mu^{h}_{t}=\operatorname{Law}\left(  g^{h}\left(  t\right) \right)$ be a probability measure on $G$. We refer to $\mu_{t}$ as the time $t$ \emph{ horizontal heat kernel measure on }$G$.

As the proof of \cite[Theorem 8.1]{DriverGordina2008} explains, in this case Assumptions \ref{a.1}, \ref{a.2} and \ref{a.3} are satisfied, and moreover, \eqref{e.4.4} is satisfied as follows. Namely, \cite[Corollary 7.3]{DriverGordina2008} says that the Ricci curvature is bounded from below by $k\left( \omega \right)$ uniformly for all $G_{n}$, so \eqref{e.4.4} holds as follows.

\begin{equation}
\left\Vert J_{k}^{n}\right\Vert_{L^{p}\left(  \mu^{n}\right)  } \leqslant\exp\left(  \frac{c\left(  k \left(  \omega\right)  t\right)  \left( p-1\right)  }{2t}d_{n}^{2}\left(  \mathbf{e}, k \right)  \right), k \in G_{n},
\end{equation}
where $c\left( \cdot \right)$ is defined by \eqref{e.5.8}.

In the sub-Riemannian case we have

\begin{equation}
\left\Vert J_{k}^{h, n}\right\Vert_{L^{p}\left(  \mu^{n}_{h}\right)  } \leqslant
	\exp \left(\left(1 + \frac{8\|\omega\|_{2, n}^2}{\rho_{2, n}} \right) \frac{(1+p)\left(d_{n}^{h}(e,k)\right)^2}{4t} \right),
\end{equation}
where the geometric constants are defined as in \cite[p. 25]{BaudoinGordinaMelcher2013}.
\bibliographystyle{plain}	% (uses file "plain.bst")

\begin{thebibliography}{10}

\bibitem{BaudoinGordinaMelcher2013}
Fabrice Baudoin, Maria Gordina, and Tai Melcher.
\newblock Quasi-invariance for heat kernel measures on sub-{R}iemannian
  infinite-dimensional {H}eisenberg groups.
\newblock {\em Trans. Amer. Math. Soc.}, 365(8):4313--4350, 2013.

\bibitem{Becker1998}
Howard Becker.
\newblock Polish group actions: dichotomies and generalized elementary
  embeddings.
\newblock {\em J. Amer. Math. Soc.}, 11(2):397--449, 1998.

\bibitem{BeckerKechris1993}
Howard Becker and Alexander~S. Kechris.
\newblock Borel actions of {P}olish groups.
\newblock {\em Bull. Amer. Math. Soc. (N.S.)}, 28(2):334--341, 1993.

\bibitem{BogachevGaussianMeasures}
Vladimir~I. Bogachev.
\newblock {\em Gaussian measures}, volume~62 of {\em Mathematical Surveys and
  Monographs}.
\newblock American Mathematical Society, Providence, RI, 1998.

\bibitem{BuehlerPhD}
Theo B{\"u}hler.
\newblock {\em On the algebraic foundation of bounded cohomology}.
\newblock PhD thesis, ETH, 2008.

\bibitem{DaviesHeat_Kernels_and_Spectral_Theory}
E.~B. Davies.
\newblock {\em Heat kernels and spectral theory}, volume~92 of {\em Cambridge
  Tracts in Mathematics}.
\newblock Cambridge University Press, Cambridge, 1989.

\bibitem{delaHarpeLNM}
Pierre de~la Harpe.
\newblock {\em Classical {B}anach-{L}ie algebras and {B}anach-{L}ie groups of
  operators in {H}ilbert space}.
\newblock Lecture Notes in Mathematics, Vol. 285. Springer-Verlag, Berlin-New
  York, 1972.

\bibitem{DriverGordina2008}
Bruce~K. Driver and Maria Gordina.
\newblock Heat kernel analysis on infinite-dimensional {H}eisenberg groups.
\newblock {\em J. Funct. Anal.}, 255(9):2395--2461, 2008.

\bibitem{DriverGordina2009}
Bruce~K. Driver and Maria Gordina.
\newblock Integrated {H}arnack inequalities on {L}ie groups.
\newblock {\em J. Differential Geom.}, 83(3):501--550, 2009.

\bibitem{DriverHall1999a}
Bruce~K. Driver and Brian~C. Hall.
\newblock The energy representation has no non-zero fixed vectors.
\newblock In {\em Stochastic processes, physics and geometry: new interplays,
  {II} ({L}eipzig, 1999)}, volume~29 of {\em CMS Conf. Proc.}, pages 143--155.
  Amer. Math. Soc., Providence, RI, 2000.

\bibitem{Janson1997}
Svante Janson.
\newblock {\em Gaussian {H}ilbert spaces}, volume 129 of {\em Cambridge Tracts
  in Mathematics}.
\newblock Cambridge University Press, Cambridge, 1997.

\bibitem{KuoBook1975}
Hui~Hsiung Kuo.
\newblock {\em Gaussian measures in {B}anach spaces}.
\newblock Springer-Verlag, Berlin, 1975.
\newblock Lecture Notes in Mathematics, Vol. 463.

\bibitem{LiYau1986}
Peter Li and Shing-Tung Yau.
\newblock On the parabolic kernel of the {S}chr\"odinger operator.
\newblock {\em Acta Math.}, 156(3-4):153--201, 1986.

\bibitem{Melcher2009}
Tai Melcher.
\newblock Heat kernel analysis on semi-infinite {L}ie groups.
\newblock {\em J. Funct. Anal.}, 257(11):3552--3592, 2009.

\bibitem{WangFY1997a}
Feng-Yu Wang.
\newblock Logarithmic {S}obolev inequalities on noncompact {R}iemannian
  manifolds.
\newblock {\em Probab. Theory Related Fields}, 109(3):417--424, 1997.

\bibitem{WangFY2004a}
Feng-Yu Wang.
\newblock Equivalence of dimension-free {H}arnack inequality and curvature
  condition.
\newblock {\em Integral Equations Operator Theory}, 48(4):547--552, 2004.

\end{thebibliography}
\def\cprime{$'$}

\end{document}